\documentclass{amsart}
\usepackage{amsfonts}
\newtheorem{thm}{Theorem}[section]

\newtheorem{lem}[thm]{Lemma}
\newtheorem{prop}[thm]{Proposition}
\newtheorem{defn}[thm]{Definition}

\begin{document}
\title[Some remarks on semisimple Leibniz algebras]
{Some remarks on semisimple Leibniz algebras}
\author{S.  G\'{o}mez - Vidal, B.A. Omirov and A.Kh. Khudoyberdiyev}
\address{[Vidal G\'{o}mez] Dpto. Matem\'{a}tica Aplicada I.
Universidad de Sevilla. Avda. Reina Mercedes, s/n. 41012 Sevilla.
(Spain)} \email{samuel.gomezvidal@gmail.com}
\address{[B.A. Omirov -- A.Kh. Khudoyberdiyev] Institute of Mathematics and Information Technologies
 of Academy of Uzbekistan, 29, Do'rmon yo'li street., 100125, Tashkent (Uzbekistan)}
\email{omirovb@mail.ru --- khabror@mail.ru}


\begin{abstract}
From the Levi's Theorem it is known that every finite dimensional
Lie algebra over a field of characteristic zero is decomposed into
semidirect sum of solvable radical and semisimple subalgebra.
Moreover, semisimple part is the direct sum of simple ideals. In
\cite{Bar} the Levi's theorem is extended to the case of Leibniz
algebras. In the present paper we investigate the semisimple
Leibniz algebras and we show that the splitting theorem for
semisimple Leibniz algebras is not true. Moreover, we consider
some special classes of the semisimple Leibniz algebras and find a
condition under which they decompose into direct sum of simple
ideals.

\end{abstract}
 \maketitle

\textbf{Mathematics Subject Classification 2010}: 17A32, 17A60, 17B20.

\textbf{Key Words and Phrases}: Lie algebras, Leibniz algebras,
simplicity, semisimplicity, solvable radical, Levi's Theorem.

\section{Introduction}

Due to Levi's theorem the study of finite dimensional Lie algebra
over a field of characteristic zero is reduced to the study of
solvable and semisimple algebras \cite{Jac}. From results of
\cite{Mal} we can conclude that the main part of solvable Lie
algebra consists of the maximal nilpotent ideal. The
classification of semisimple Lie algebras has been known since the
works of Cartan and Killing \cite{Jac}. According to the
Cartan-Killing theory the semisimple Lie algebra can be
represented as a direct sum of simple Lie algebras.

The notion of Leibniz algebras have been first introduced by Loday
in \cite{loday} as a non-antisymmetric generalization of Lie
algebras. The last $20$ years the theory of Leibniz algebras has
been actively studied and many results of the theory of Lie
algebras have been extended to Leibniz algebras. Until now a lot
of works are devoted to the description of finite-dimensional
nilpotent Leibniz algebras  \cite{AOR1}--\cite{Ayu}. However,
simple and semisimple parts are not studied. It is because the
notions of simple and semisimple Leibniz algebras are not agreed
with corresponding the classical notions. In fact, in non-Lie
Leibniz algebra $L$ there is non-trivial ideal, which is a
subspace spanned by squares of elements of the algebra $L$
(denoted by $I$). Therefore, in \cite{Dzum} the notion of simple
Leibniz algebra has been suggested, namely, a Leibniz algebra $L$
is called {\it simple} if it contains only ideals $\{0\},$ $I,$
$L$ and square of the algebra is not equal to the ideal $I.$ In
the case when the Leibniz algebra is Lie algebra, the ideal $I$ is
trivial and this definition agrees with the definition of simple
Lie algebra. Obviously, the quotient algebra by ideal $I$ of
simple Leibniz algebra is simple Lie algebra, but the converse is
not true.

From an analogue of Levi's theorem for Leibniz algebras \cite{Bar}
the description of simple Leibniz algebras immediately follows. In
the present paper we present the same description but with another
proof. Moreover, we introduce a notion of a semisimple Leibniz
algebra (algebra whose solvable radical is coincided with $I$) and
investigate such algebras. Note that Leibniz algebra is semisimple
if and only if quotient Lie algebra is semisimple. In particular,
we find some sufficient conditions under which an analogue of
splitting theorem for semisimple Leibniz algebras is true. In
addition, an example of semisimple Leibniz algebra, which is not
decomposed into a direct sum of simple Leibniz ideals, is given.

Actually, there exist semisimple Leibniz algebras (which are not
simple in general) for which the quotient algebra is simple Lie
algebra. So, we call such algebras {\it Lie-simple} Leibniz
algebras. According to this definition the natural question arises
- whether an arbitrary finite dimensional semisimple Leibniz
algebra is a direct sum of Lie-simple Leibniz algebras. We show
that the answer to the question is also negative, we give a
counterexample.

Finally, for some special classes of semisimple Leibniz algebras
we give sufficient conditions under which these classes decomposed
into a direct sum of the Lie-simple Leibniz algebras. More
precisely, we consider a semisimple Leibniz algebra consisting of
the direct sum of the classical three dimensional simple Lie
algebras $sl_2.$

In this paper all algebras and vector spaces are considered over a field of characteristic zero and finite dimensional.

We shall use symbols: $+$, $\oplus$ and $\dot{+}$ for notations of
direct sum of vector spaces,  direct and semidirect sums of
algebras, respectively.

\section{Preliminaries}

In this section we give necessary definitions and preliminary results.

\begin{defn} An algebra $(L,[\cdot,\cdot])$ over a field $F$ is called a Leibniz algebra if for any $x,y,z\in L$ the so-called Leibniz identity
$$[x,[y,z]]=[[x,y],z] - [[x,z],y]$$ holds true.
\end{defn}

For a given Leibniz algebra $L$ we define derived sequence as follows:
$$
L^1=L,\quad L^{[n+1]}=[L^{[n]},L^{[n]}], \quad n \geq 1.
$$

\begin{defn} A Leibniz algebra $L$ is called
solvable, if there exists  $m\in\mathbb N$ such that $L^{[m]}=0.$
The minimal number $m$ with this property is called index of
solvability of the algebra $L.$
\end{defn}

Let us recall Levi's theorem for Lie algebras.

\begin{thm}\cite{Jac}. For an arbitrary finite dimensional Lie algebra  $B$
over a field of characteristic zero with solvable radical $R,$
there exists semisimple subalgebra $G$ such that $B = G \dot{+}R.$
\end{thm}

Further we shall need in the following splitting theorem for
semisimple Lie algebras.

\begin{thm}\cite{Jac}. \label{thm24} An arbitrary finite dimensional semisimple Lie algebra is
decomposed into a direct sum of simple ideals and the decomposition is unique up to permutations of summands.
\end{thm}

In \cite{Bar} it is shown that for Leibniz algebra $L$ the ideal $I= id < [x,x] \, |\, x\in L >$ coincides with the space spanned by squares of elements of $L.$ Note that the ideal $I$ is the minimal ideal with respect to the property that the quotient algebra $L/I$ is a Lie algebra.

According to \cite{Jac} a three-dimensional simple
Lie algebra is said to be {\it split} if the algebra contains an element $h$
such that $ad \ h$ has a non-zero characteristic root $\rho$
belonging to the base field. It is well-known that any such
algebra has a basis $\{e,f,h\}$ with the multiplication table
$$ [e,h]=2e, \qquad [f,h]=-2f, \qquad [e,f]=h,$$
$$ [h,e]=-2e, \qquad [h,f]=2f, \qquad [f,e]=-h.$$

This simple $3$-dimensional Lie algebra denoted by $sl_2$ and the
basis $\{e,f,h\}$ is called {\it canonical basis}. Note that any
$3$-dimensional simple Lie algebra is isomorphic to $sl_2.$

Here is the result of \cite{Rakh} which describes simple Leibniz algebras whose quotient Lie algebras are isomorphic to $sl_2.$
\begin{thm}\label{t1}
Let $L$  be a complex finite dimensional simple Leibniz
algebra. Assume that the quotient Lie algebra $L/I$ is isomorphic to the
algebra $sl_2.$ Then there exist a basis $\{e,f,h,x_0,x_1,\dots,
x_m\}$ of $L$ such that non-zero products of basis elements in $L$
are represented as follows:
$$\begin{array}{lll} \, [e,h]=2e, & [h,f]=2f, &[e,f]=h, \\
\, [h,e]=-2e,& [f,h]=-2f, & [f,e]=-h,\\
\, [x_k,h]=(m-2k)x_k, & 0 \leq k \leq m ,&\\
\, [x_k,f]=x_{k+1},  & 0 \leq k \leq m-1, & \\
\, [x_k,e]=-k(m+1-k)x_{k-1}, & 1 \leq k \leq m. &\\
 \end{array}$$
\end{thm}

In \cite{Rakh} Leibniz algebras (they are not necessary to be simple) for which  the quotient Lie algebras are isomorphic to $sl_2$ are described.
Let us present a Leibniz algebra $L$ with table of multiplication in a basis $\{e, f, h, x_1^j,\dots,x_{t_j}^j, \, 1\leq j \leq p\}$
which is not simple, but the quotient algebra $L/I$ is a simple
\cite{Rakh}:
$$\begin{array}{lll} \, [e,h]=2e, & [h,f]=2f, &[e,f]=h, \\
\, [h,e]=-2e,& [f,h]=-2f, & [f,e]=-h,\\
\, [x_k^{j},h]=(t_j-2k)x_k^{j}, & 0 \leq k \leq t_j ,&\\
\, [x_k^{j},f]=x_{k+1}^{j},  & 0 \leq k \leq t_j-1, & \\
\, [x_k^{j},e]=-k(t_j+1-k)x_{k-1}^{j}, & 1 \leq k \leq t_j. &\\
 \end{array}$$
where $L=sl_2 + I_1+I_2+ \dots +I_p$ and $I_j=\langle
x_1^j,\dots,x_{t_j}^j \rangle, \, 1\leq j \leq p.$

Note that the above presented Leibniz algebras are examples of non-simple but Lie-simple Leibniz algebras.

Now we define a notion of semisimple Leibniz algebra.
\begin{defn} A Leibniz algebra L is called semisimple if its
maximal solvable ideal is equal to $I$. \end{defn} Since in Lie
algebras case the ideal $I$ is equal to zero, this definition also
agrees with the definition of semisimple Lie algebra.

Obviously, for the sets of $n$-dimensional simple ($SimpL_n$),
Lie-simple ($LieSimpL_n$) and semisimple ($SemiSimpL_n$) Leibniz
algebras the following embeddings  are true:
$$SimpL_n \subseteq LieSimpL_n \subseteq SemiSimpL_n.$$

Although the Levi's theorem is proved for the left Leibniz
algebras \cite{Bar} it also is true for the right Leibniz
algebras (we considering right Leibniz algebras).

\begin{thm}\cite{Bar} (Levi's Theorem).\label{t2}
Let $L$ be a finite dimensional Leibniz algebra over a field of
characteristic zero and $R$ be its solvable radical. Then there
exists a semisimple subalgebra $S$ of $L$, such that
$L=S\dot{+}R.$
\end{thm}

From the proof of Theorem \ref{t2} it is not difficult to see that
$S$ is a semisimple Lie algebra. Therefore, we have that a simple
Leibniz algebra is a semidirect sum of simple Lie algebra $S$ and
irreducible right module $I,$ i.e. $L=S\dot{+}I.$ Hence, we get
the description of simple Leibniz algebras in terms of simple Lie
algebras and its ideals $I$. For example see the algebras of
Theorem \ref{t1}.

\begin{defn} A nonzero module $M$ whose only submodules are the
module itself and zero module is called {\it irreducible module}. A
nonzero module $M$ which is a direct sum of irreducible modules is
said to be {\it completely reducible}.
\end{defn}

Further we shall use the following classical result of the theory of Lie algebras.

\begin{thm}\cite{Jac}\label{mod1} Let G be a semisimple Lie algebra over a field of characteristic zero.
Then every finite dimensional module over $G$ is completely
reducible
\end{thm}

Here is an example of simple Leibniz algebras constructed in \cite{Dzum}.

\textbf{Example 1.} Let $G$ be a simple Lie algebra
and $M$ be an irreducible skew-symmetric $G-$module (i.e.
$[x,m]=0$ for all $x \in G, m \in M$). Then the vector space
$Q=G+M$ equipped with the multiplication $[x+m,y+n]=[x,y]+[m,y],$
where $m, n \in M, x, y \in G$ is a simple Leibniz algebra.

\section{The main results}

As it was mentioned above from Theorem \ref{t2}\ it follows that
any simple Leibniz algebra is presented as a semidirect sum of a
simple Lie algebra and the ideal $I.$

Below we give another proof of the description of simple Leibniz algebras without using Levi's theorem.

\begin{thm} Let L be a finite dimensional simple Leibniz algebra. Then it has the construction
of the Example 1 for $G \cong L/I$ and $M=I.$
\end{thm}
\begin{proof}
Let $L$ be an algebra satisfying the conditions of theorem. It should be noted, that the ideal $I$ may be considered as a right $L/I$-module by the action: $$m*(a+I)=[m,a],$$ where $m \in
I, a+I\in L/I.$ Since $L/I$ is a simple Lie algebra then by Whytehead's Lemma \cite{Jac} we have $H^2(L/I, I)=0.$

J.-L. Loday \cite{lodpir} established that there exist the
following natural bijection: $$Ext(L,M) \cong HL^2(L,M).$$

Recall a result of T. Pirashvili \cite{pir}, which says the
following:

Let $g$ is semisimple Lie algebra and $M$ is right irreducible
module over $g$ such that $H^2(g,M)=0.$ Then $HL^2(g, M)=0.$

Let now $L=L' \oplus I$ is a direct sum vector spaces, where $L' \cong L/I$ is a simple Lie
algebra. Since ideal $I$ is contained in right annihilator of the algebra $L$ then we have $[L', I]=0$ and $[I , I]=0.$ Due to simplicity of Leibniz algebra $L$ we derive that the ideal $I$ is irreducible $L/I$-module. Using Pirashvili's result we have $HL^2(L/I,I)=0,$ but the condition
$0=HL^2(L/I,I) \cong Ext(L/ I,I)$ is equivalent to $L \cong L/I
\dot{+} I.$ \end{proof}

Let us investigate the case of semisimple Leibniz algebras. Let
$L$ be a semisimple Leibniz algebra. Similarly to the case of
simple Leibniz algebras we can establish that $L \cong L/I + I.$
It is known the result on decomposition of semisimple Lie algebra
$L/I$ into a direct sum of simple Lie ideals. Moreover, we have
that $L/I$-module $I$ is completely reducible and hance, ideal $I$
is decomposed into a direct sum of irreducible submodules over the
Lie algebra $L/I.$

Taking into account these results for semisimple Leibniz algebras
it seems that the following conclusion is true for Leibniz
algebras case:

\textbf{Conclusion.} An arbitrary finite dimensional semisimple
Leibniz algebra is decomposed into direct sum of simple ones.

Let $L$ be a finite dimensional semisimple Leibniz algebra. Then
according to Theorem \ref{t2} we have $L = S \dot{+} I,$ where $S$
is a semisimple Lie algebra and $[I, S] = I.$ From Theorem
\ref{thm24} we get $S = S_1 \oplus S_2 \oplus \dots \oplus S_k,$
where $S_i$, ($\ 1 \leq i \leq k$) is a simple Lie algebra. Thus,
we have
$$L = (S_1 \oplus S_2 \oplus \dots \oplus S_k) \dot{+} I.$$

Let us introduce the denotation $I_j = [I, S_j]$ for $1 \leq j \leq k.$

\begin{lem}\label{l1} The following are true:

a) $I = I_1+ I_2+ \dots + I_k;$

b) $I_j$ is an ideal of $L$ for all $ j \ (1 \leq j \leq k);$

c) $I_j = [I_j, S_j]$ for all $ j \ (1 \leq j \leq k);$

d) $S_j + I_j$ is an ideal of $L$ for all $ j \ (1 \leq j \leq k)þ$

\end{lem}
\begin{proof} Since $I$ is an ideal of $L$ then $I_j = [I, S_j] \subseteq I$ for all $j.$ Hence, $I_1 + I_2+ \dots
+ I_k \subseteq I.$

From $$ I = [I, S] = [I, S_1 \oplus S_2 \oplus
\dots \oplus S_k] \subseteq [I, S_1] + [I, S_2]+ \dots +[I, S_k] =
I_1+I_2+ \dots +I_k,$$ we have the correctness of the statement a).

The proof of the statement b) follows from property $[S_j, S_j]
=S_j$ and we have
$$[I_j, L] = [[I, S_j], L] \subseteq [I, [ S_j, L]] + [[I, L],
S_j] = $$$$=[I, [ S_j, (S_1 \oplus S_2\oplus \dots \oplus S_k)
\dot{+}I]] + [[I, (S_1 \oplus S_2 \oplus \dots \oplus S_k) \dot{+}
I], S_j]\subseteq$$
$$\subseteq [I, [ S_j, (S_1 \oplus S_2\oplus \dots \oplus S_k)]] + [[I, (S_1 \oplus S_2 \oplus \dots \oplus S_k)],
S_j]\subseteq$$ $$ \subseteq [I, [ S_j, S_j]] + [I, S_j] \subseteq
[I, S_j] = I_j.$$

Since from b) we have that $I_j$ is an ideal of $L,$ we obtain $[I_j,
S_j] \subseteq I_j$ and from $$I_j = [I, S_j] = [I, [S_j, S_j]]
\subseteq [[I, S_j], S_j] + [[I, S_j], S_j] = [I_j, S_j],$$ we get
$[I_j, S_j] = I_j.$ So, the statement c) is also proved.

From the following equalities:  $$[S_j+I_j, L] = [S_j+I_j,
(S_1\oplus  S_2 \oplus \dots \oplus S_k) \dot{+} I] =$$ $$= [S_j,
S_1\oplus  S_2 \oplus \dots \oplus S_k] + [I_j, S_1\oplus  S_2
\oplus \dots \oplus S_k]= S_j+I_j,$$
$$[L, S_j+I_j] = [(S_1\oplus  S_2 \oplus \dots \oplus S_k) \dot{+} I, S_j] = [S_1\oplus  S_2 \oplus \dots \oplus S_k, S_j] + [I,
S_j] = S_j+I_j,$$

we get that $S_j+I_j$ is an ideal of $L.$ Thus, the part d) is also proved.
\end{proof}

The following example shows that the Conclusion is not true in
general.

\textbf{Example 2.} Let $L$ be a semisimple Leibniz algebra such
that $L=(sl_2^1 \oplus sl_2^2) \dot{+} I,$ where $I= I_1 \oplus
I_2$ and $[I_1, sl_2^2] = [I_2, sl_2^1] =0.$ Moreover, $I_1 =
I_{1,1}\oplus I_{1,2}, I_2 = I_{2,1}\oplus I_{2,2},$ where
$I_{1,1}$ and $I_{1,2}$ are irreducible $sl_2^1$-modules.
Respectively, $I_{2,1}$ and $I_{2,2}$ are irreducible
$sl_2^2$-modules. Then, using the result \cite{Rakh} we conclude
that, there exists a basis $\{e_1, h_1, f_1, e_2, h_2, f_2,
x_0^{j}, x_1^{j}, \dots ,x_{t_j}^{j}\},$ ($1\leq j \leq 4$) such
that multiplication table of $L$ in this basis has the following
form:

$$[sl_2^i, sl_2^i]: \quad \begin{array}{lll}
\ [e_i,h_i]=2e_i, & [f_i,h_i]=-2f_i, & [e_i,f_i]=h_i, \\
\ [h_i,e_i]=-2e_i & [h_i,f_i]=2f_i,  & [f_i,e_i]=-h_i, \ i=1,2. \\
 \end{array}$$

$$[I_1,sl_2^1]: \quad \begin{array}{ll}
\, [x_k^j,h_1]=(t_j-2k)x_k^j, & 0 \leq k \leq t_j, \\
\, [x_k^j,f_1]=x_{k+1}^j,  & 0 \leq k \leq t_j-1, \\
\, [x_k^j,e_1]=-k(t_j+1-k)x_{k-1}^j, & 1 \leq k \leq t_j, \ j=1,2.\\
 \end{array}$$

$$[I_2,sl_2^2]: \quad \begin{array}{ll}
\, [x_k^j,h_2]=(t_j-2k)x_k^j, & 0 \leq k \leq t_j, \\
\, [x_k^j,f_2]=x_{k+1}^j,  & 0 \leq k \leq t_j-1, \\
\, [x_k^j,e_2]=-k(t_j+1-k)x_{k-1}^j, & 1 \leq k \leq t_j, \ j=3,4.\\
 \end{array}$$
where $I_{1,1}=\{ x_0^1,\dots,x_{t_1}^1\},$ $I_{1,2}=\{
x_0^2,\dots,x_{t_2}^2\},$ $I_{2,1}=\{ x_0^3,\dots,x_{t_3}^3\}$ and
$I_{2,2}=\{x_0^4,\dots,x_{t_4}^4\}.$

Evidently, the algebra $L$ is semisimple and it is decomposed into
the direct sum of two ideals $sl_2^1\dot{+}(I_{1,1}\oplus
I_{1,2})$ and $sl_2^2\dot{+}(I_{2,1}\oplus I_{2,2}),$ which are
not simple Leibniz algebras, but there are Lie-simple Leibniz
algebras.

Example 2 shows that in case of $I_j$ is reducible module over a simple Lie algebra $S_j,$ then
the Conclusion is not true. Now we consider the case of $I_j$ is
an irreducible module. First we prove the following lemma.

\begin{lem}\label{l2} Let $L$ be a semisimple Leibniz algebra such that $L=(sl_2\oplus S) \dot{+}I,$ where $S$ is an arbitrary simple Lie algebra. Let $I$ is irreducible over $sl_2,$ then $[I,S] = 0.$
\end{lem}
\begin{proof}
Let $dimI=m+1,$ then similarly as in the proof of Theorem \ref{t1} we have the existence of basis $\{e, f, h, x_0, x_1, \dots , x_m\}$ of $sl_2+I$
such that table of multiplication has the following form:
$$\begin{array}{ll}
\, [x_i,h]=(m-2i)x_i, & 0 \leq i \leq m, \\
\, [x_i,f]=x_{i+1},  & 0 \leq i \leq m-1, \\
\, [x_i,e]=-i(m+1-i)x_{k-1}, & 1 \leq i \leq m. \\
 \end{array}$$

Let $\{y_1, y_2, \dots, y_n\}$ is a basis of algebra $S.$ We set
$$[x_0, y_j] = \sum\limits_{k=0}^m \alpha_{k,j}x_k, \quad 1 \leq j \leq n.$$

Consider the Leibniz identity $$[[x_0, y_j], f] = [x_0, [y_j, f]]
+ [[x_0, f], y_j] = [x_1, y_j].$$

On the other hand $$[[x_0, y_j], f] = [\sum\limits_{k=0}^m
\alpha_{k,j}x_k, f] = \sum\limits_{k=0}^{m-1}
\alpha_{k,j}x_{k+1}.$$

Hence, we get
$$[x_1, y_j] = \sum\limits_{k=0}^{m-1} \alpha_{k,j}x_{k+1} , \quad 1 \leq j \leq n.$$

From the equalities $[[x_i, y_j], f] = [x_i, [y_j, f]] + [[x_i,
f],y_j],$ we obtain $$[x_i, y_j] = \sum\limits_{k=0}^{m-i}
\alpha_{k,j}x_{k+i} , \quad 0 \leq i \leq m, \ 1 \leq j \leq n.$$

Consider the Leibniz identity $$[[x_0, y_j], e] = [x_0, [y_j, e]]
+ [[x_0, e], y_j] = 0.$$

On the other hand, $$[[x_0, y_j], e] = [\sum\limits_{k=0}^m
\alpha_{k,j}x_k, e] = \sum\limits_{k=1}^{m}
k(-m-1+k)\alpha_{k,j}x_{k-1}.$$ We obtain $\alpha_{i,j}
=0,$ for $1 \leq i \leq m, \ 1 \leq j \leq n$ and we can assume
that $$[x_i, y_j] = \alpha_jx_i, \quad 0 \leq i \leq m, \ 1 \leq j
\leq n.$$

Using the Leibniz identity, we have
$$[x_i,[y_j, y_k]]= [[x_i,y_j], y_k] - [[x_i,y_k], y_j] =
[\alpha_jx_i, y_k] - [\alpha_kx_i, y_j] = \alpha_k\alpha_kx_i -
\alpha_k\alpha_kx_i =0.$$

Taking into account the property $[S,S] =S$ and
arbitrariness of elements $\{x_i, y_j, y_k\}$ we get $[I,[S,S]] =
[I,S] = 0.$
\end{proof}

In the following Theorem we show the trueness of the Conclusion under some conditions.

\begin{thm}
Let $L$ be a semisimple Leibniz algebras such that
$L=(sl_2^1\oplus sl_2^2 \oplus \dots \oplus sl_2^{k-1} \oplus S_k)
\dot{+}I.$ Let $I_j$ is irreducible module over $sl_2^j$ for
$j=\overline{1, k-1}$ and $I_k$ is irreducible over $S_k.$ Then
$L$ is decomposed into direct sum of simple Leibniz algebras,
namely,
$$L=(sl_2^1\dot{+}I_1)\oplus (sl_2^2\dot{+}I_2)\oplus \dots \oplus (sl_2^{k-1}\dot{+}I_{k-1})\oplus (S_k \dot{+}I_k).$$
\end{thm}
\begin{proof} By Lemma \ref{l1} it is known that $sl_2^j+I_j$ and
$S_k+I_k$ are ideals of $L.$ Since $I_j$ is irreducible over
$sl_2^j,$ then by Lemma \ref{l2} we obtain $[I_j, sl_2^i]=0$ for all $i,j \ (i \neq j).$

Thus, we have $$[sl_2^i+I_i, sl_2^j+I_j] = 0, \quad i \neq j,
\ 1 \leq i,j \leq k-1.$$

Moreover, from Lemma \ref{l2} we have $[I_i, S_k+I_k] = 0$ for
$1 \leq i \leq k-1.$ In order to complete the proof of theorem it is
necessary to establish the equality $[I_k, sl_2^j]=0.$

Let us assume the contrary, i.e. $[I_k, sl_2^j] \neq 0,$ for some
$j \ (1 \leq j \leq k-1).$ Since $I_k$ is an ideal of $L$ and $I_k
\subseteq I,$ then we have $[I_k, sl_2^j]\subseteq I_k.$

From
$$[I_k, sl_2^j] = [[I,S_k], sl_2^j]\subseteq [I,[S_k, sl_2^j]] +
[[I,sl_2^j],S_k] = [I_j,S_k] \subseteq I_j,$$ we obtain $[I_k,
sl_2^j]\subseteq I_j.$

Hence, we get $I_k\cap I_j \neq 0.$ Since $I_k\cap I_j$ is an ideal
of the algebra $L,$ then it can be considered as the right module over
$sl_2^j$ and $S_k.$ Due to $I_j$ and $I_k$ are irreducible, we have
$I_k \cap I_j = I_j = I_k.$ From $[I_j, S_k] = 0$ we derive $[I_k,
S_k] = 0,$ but it is a contradiction with condition
$$[I_k, S_k] = I_k.$$
Therefore, we have $$[I_k, sl_2^j] = 0.$$

Thus, we get that $[S_k+ I_k, sl_2^j+I_j] = [sl_2^j+I_j, S_k+ I_k]
= 0,$ which leads that the Leibniz algebra $L$ is decomposed into
direct sum of simple ideals.
\end{proof}

In Example 2, it is shown that if $I_j$ is reducible over $S_j,$
then the semisimple Leibniz algebra is not decomposable to direct
sum of simple ideals. However this algebra is decomposed into
direct sum of Lie-simple algebras.

Naturally arises a question: whether any semisimple Leibniz
algebras can be represented as a direct sum of Lie-simple Leibniz
algebras.

The following example gives the negative answer to this question.

\textbf{Example 3.} Let $L$ be a $10-$ dimensional semisimple
Leibniz algebra. Let $\{e_1, h_1, f_1, e_2, h_2, f_2, x_1, x_2,\\
x_3, x_4\}$ be a basis of the algebra $L$ such that $I=\{x_1, x_2, x_3, x_4\},$
and multiplication table of $L$ has the following
form:
$$[sl_2^i, sl_2^i]: \quad \begin{array}{lll}
\ [e_i,h_i]=2e_i, & [f_i,h_i]=-2f_i, & [e_i,f_i]=h_i, \\
\ [h_i,e_i]=-2e_i & [h_i,f_i]=2f_i,  & [f_i,e_i]=-h_i, \ i=1, 2, \\
 \end{array}$$
$$[I,sl_2^1]: \quad \begin{array}{llll}
\, [x_1,f_1]=x_2, & [x_1,h_1]=x_1, & [x_2,e_1]=-x_1, & [x_2,h_1]=-x_2, \\
\, [x_3,f_1]=x_4, & [x_3,h_1]=x_3, & [x_4,e_1]=-x_3, & [x_4,h_1]=-x_4, \\
 \end{array}$$
$$[I,sl_2^2]: \quad \begin{array}{llll}
\, [x_1,f_2]=x_3, & [x_1,h_2]=x_1, & [x_3,e_2]=-x_1, & [x_3,h_2]=-x_3, \\
\, [x_2,f_2]=x_4, & [x_2,h_2]=x_2, & [x_4,e_2]=-x_2, & [x_4,h_2]=-x_4, \\
 \end{array}$$
(omitted products are equal to zero).

From this table of multiplications we have $[I,sl_2^1] =[I,sl_2^2] =I.$
Moreover, $I$ splits over $sl_2^1$ (i.e. $I=\{x_1,x_2\}\oplus \{x_3,x_4\}$) and
over $sl_2^2$ (i.e. $I=\{x_1,x_3\}\oplus \{x_2,x_4\}$). Therefore, $$L=(sl_2^1\oplus
sl^2_2)\dot{+}I \neq (sl_2^1\dot{+} I_1)\oplus
(sl^2_2\dot{+}I_2).$$

Below we find some types of semisimple Leibniz algebras which are
decomposed into a direct sum of Lie-simple Leibniz algebras.

Let $L$ be a semisimple Leibniz algebra such that $L=(sl_2\oplus
S) \dot{+} I,$ where $S$ is a simple Lie algebra. Let $I_1 = [I,
sl_2]$ be an irreducible over $sl_2,$ then according to Theorem
\ref{mod1} the module $I_1$ is a completely reducible, i.e.
$I_1=I_{1,1}\oplus I_{1,2} \oplus \dots \oplus I_{1,p},$ where $
I_{1,i}$ is irreducible over $sl_2^1.$

Let us consider the case of $p=2.$
\begin{prop}\label{prop1}
If $dim I_{1,1} \neq dim I_{1,2},$ then $L= (sl_2\dot{+} I_1)
\oplus (S \dot{+} I_2).$
\end{prop}
\begin{proof}
Let $I_1=I_{1,1}\oplus I_{1,2},$ then there exists a basis $\{e, h,
f, x_0^{1}, x_1^{1}, \dots ,x_{t_1}^{1}, x_0^{2}, x_1^{2}, \dots
,x_{t_2}^{2}\}$ of $sl_2+I_1,$ such that

$$[I_1,sl_2]: \quad \begin{array}{ll}
\, [x_k^j,h]=(t_j-2k)x_k^j, & 0 \leq k \leq t_j, \\
\, [x_k^j,f]=x_{k+1}^j,  & 0 \leq k \leq t_j-1, \\
\, [x_k^j,e]=-k(t_j+1-k)x_{k-1}^j, & 1 \leq k \leq t_j, \ j=1, 2.\\
 \end{array}$$
where $I_{1,1}=\{x_0^1,\dots,x_{t_1}^1\},$ $I_{1,2}=\{
x_0^2,\dots,x_{t_2}^2\}.$

Without loss of generality we can assume that $t_1 > t_2.$

Let $\{y_1, y_2, \dots, y_m\}$ be a basis of the algebra $S.$

We put $$[x_0^i, y_1]
= \sum\limits_{j=1}^2\sum\limits_{r=0}^{t_j} \alpha_{i,r}^jx_r^j,
\quad 1 \leq i \leq 2.$$

Consider equalities $$[[x^1_0, f],y_1] = [x_0^1, [f,
y_1]] + [[x_0^1, y_1], f] =
[\sum\limits_{j=1}^2\sum\limits_{r=0}^{t_j} \alpha_{1,r}^jx_r^j,
f] = \sum\limits_{j=1}^2\sum\limits_{r=0}^{t_j-1}
\alpha_{1,r}^jx_{r+1}^j.$$

On the other hand $$[[x^1_0, f],y_1] = [x^1_1,y_1].$$

Hence, we get
$$[x_1^1, y_1] = \sum\limits_{j=1}^2\sum\limits_{r=0}^{t_j-1}
\alpha_{1,r}^jx_{r+1}^j.$$

From equalities
$$[x_k^1, y_1]=[[x^1_{k-1}, f],y_1] = [x_{k-1}^1, [f,
y_1]] + [[x_{k-1}^1, y_1], f]$$ and induction we derive
$$[x_k^1, y_j] = \sum\limits_{j=1}^2\sum\limits_{r=0}^{t_j-k}
\alpha_{1,r}^jx_{r+k}^j, \quad 1 \leq k \leq t_2,$$
$$[x_k^1, y_j] = \sum\limits_{r=0}^{t_1-k}
\alpha_{1,r}^jx_{r+k}^j, \quad t_2+1 \leq k \leq t_1.$$

Consider the products $$[[x^1_0, e],y_1] = [x_0^1, [e, y_1]] +
[[x_0^1, y_1], e] = [[x_0^1, y_1], e]=
$$ $$=[\sum\limits_{j=1}^2\sum\limits_{r=0}^{t_j} \alpha_{1,r}^jx_r^j,
e] = \sum\limits_{j=1}^2\sum\limits_{r=1}^{t_j}
\alpha_{1,r}^j(-r(t_j+1-r))x_{r-1}^j.$$

On the other hand $$[[x^1_0, e],y_1] = 0.$$

Comparing the coefficients at the basis elements, we obtain
$$\alpha_{1,r}^j=0, \quad 1 \leq j \leq 2, \ 1\leq r \leq t_j.$$

Thus, we have
$$[x_k^1, y_1] = \alpha_{1,0}^1x_{k}^1+\alpha_{1,0}^2x_{k}^2, \quad 0 \leq k \leq t_2,$$
$$[x_k^1, y_1] = \alpha_{1,0}^1x_{k}^1, \quad t_2+1 \leq k \leq t_1.$$

Similarly, from the products $[[x_k^2,f],y_1]$ for $0 \leq k \leq t_2-1$ and $[[x_0^2,e],y_1]$
we obtain
$$[x_k^2, y_1] = \alpha_{2,0}^1x_{k}^1+\alpha_{2,0}^2x_{k}^2, \quad 0 \leq k \leq t_2,$$

Consider
$$[[x^2_{t_2}, f],y_1] = [x^2_{t_2}, [f, y_1]]
+ [[x^2_{t_2}, y_1], f] =
[\alpha_{2,0}^1x_{t_2}^1+\alpha_{2,0}^2x_{t_2}^2, f] =
\alpha_{2,0}^1x_{t_2+1}^1.$$

From the equality $[x^2_{t_2}, f] = 0$ we obtain
$\alpha_{2,0}^1=0.$ Thus, we get
$$[x_k^2, y_1] = \alpha_{2,0}^2x_{k}^2, \quad 0 \leq k \leq t_2.$$

Consider the products $$[[x^1_0, h],y_1] = [x^1_0, [h, y_1]]
+ [[x^1_0, y_1], h] = [\alpha_{1,0}^1x^1_0+\alpha_{1,0}^2x^2_0, h]
= t_1\alpha_{1,0}^1x^1_0 + t_2\alpha_{1,0}^2x^2_0.$$

On the other hand $$[[x^1_0, h],y_1] = [t_1x^1_0, y_1] =
t_1\alpha_{1,0}^1x^1_0 + t_1\alpha_{1,0}^2x^2_0.$$

Comparing the coefficient at the basis elements, we have $(t_1 -
t_2)\alpha_{1,2}=0.$ The condition $t_1 \neq t_2,$
implies $\alpha_{1,2}=0.$

Thus, we can assume that
$$[x_k^1, y_1] = \alpha_{1,1}x_{k}^1, \quad 0 \leq k \leq t_1,$$
$$[x_k^2, y_1] = \alpha_{2,1}x_{k}^2, \quad 0 \leq k \leq t_2.$$

In a similar way as above we obtain
$$[x_k^i, y_j] = \alpha_{i,j}x_{k}^i, \quad 1 \leq i \leq 2, \ 0 \leq k \leq t_i, \ 1 \leq j \leq m.$$

Consider the equalities
$$[x_k^i,[y_p, y_q]] = [[x_k^i,y_p], y_q] - [[x_k^i,y_q], y_p] =
[\alpha_{i,p}x_k^i,y_q] -
[\alpha_{i,q}x_k^i,y_p]=\alpha_{i,p}\alpha_{i,q}x_k^i -
\alpha_{i,q}\alpha_{i,p}x_k^i =0.$$

Taking into account the property $[S,S] =S$ and
arbitrariness of the elements $\{x_k^i, y_p, y_q\}$ we obtain $[I,[S,S]] = [I,S] =
0.$

Moreover, $[I_2, sl_2]=0.$ Indeed,
$$[I_2, sl_2] = [[I_2,S], sl_2]\subseteq [I_2, [S, sl_2]] + [[I_2,sl_2], S]
= [[I_2,sl_2], S] \subseteq [I_1, S] =0.$$

Thus, we have proved that the semisimple Leibniz algebra $L$ is
decomposed into the direct sum of two Lie-simple Leibniz algebras,
i.e. $L=(sl_2\dot{+} I_1)\oplus (S\dot{+}I_2).$
\end{proof}

Let $L$ be a semisimple Leibniz algebra such that $L=(sl_2^1
\oplus sl_2^2 \oplus \dots sl_2^1 \oplus sl_2^{k-1} \oplus
S_k)\dot{+}I$ and $I_j$ is a reducible module over $sl_2^j.$ Then
the module $I_j$ is a completely reducible over $sl_2^j$, i.e.
$I_j=I_{j,1}\oplus I_{j,2} \oplus \dots \oplus I_{j,p_j},$ where $
I_{j,i}$ is an irreducible over $sl_2^j.$

We generalize the above Proposition \ref{prop1} and define types
of semisimple Leibniz algebras which are decomposed into direct
sum of Lie-simple ones.

\begin{thm}\label{t3} If $dim I_{j,r} \neq dim I_{j,q}$ for any $1 \leq j \leq k-1, \ 1 \leq r, q \leq
p_j, \ p\neq q$ then $$L = (sl_2^1\dot{+}I_1)\oplus (sl_2^2\dot{+}I_2)\oplus
\dots \oplus (sl_2^{k-1}\dot{+}I_{k-1})\oplus (S_k \dot{+} I_k).$$
\end{thm}
\begin{proof}

In order to prove theorem it is sufficient to prove
$$[sl_2^{j}+I_{j}, S_k + I_k] = [S_k + I_k, sl_2^{j}+I_{j}] = 0.$$
Without loss of generality, we can suppose $j=1$ and
$I_1=I_{1,1}\oplus I_{1,2}\oplus \dots \oplus I_{1,p}.$ Let $\{e,
h, f, x_0^{1}, x_1^{1}, \dots ,x_{t_1}^{1}, x_0^{2}, x_1^{2},
\dots ,x_{t_2}^{2}, \dots x_0^{p}, x_1^{p}, \dots ,x_{t_p}^{p}\}$
be a basis of $sl_2^1+I_1$ such that
$$\begin{array}{ll}
\, [x_k^j,h]=(t_j-2k)x_k^j, & 0 \leq k \leq t_j, \\
\, [x_k^j,f]=x_{k+1}^j,  & 0 \leq k \leq t_j-1, \\
\, [x_k^j,e]=-k(t_j+1-k)x_{k-1}^j, & 1 \leq k \leq t_j,\\
 \end{array}$$
where $I_{1,j}=\{x_0^j, x_0^j, \dots,x_{t_j}^j\},$ $1 \leq j \leq
p.$

Since $dim I_{1,r} \neq dim I_{1,q},$ then without loss of
generality we can assume that $t_1
> t_2 > \dots > t_p.$

Let $\{y_1, y_2, \dots, y_m\}$ be a basis of $S_k.$  Put
$$[x_0^i, y_1] = \sum\limits_{j=1}^p\sum\limits_{r=0}^{t_j} \alpha_{i,r}^jx_r^j, \quad 1 \leq i \leq p.$$

Similarly as in the proof of the Proposition \ref{prop1}
considering Leibniz identity
$$[[x^1_k, f],y_1] = [x_k^1, [f, y_1]] + [[x_k^1, y_1], f],$$ we obtain
$$[x_k^1, y_1] = \sum\limits_{j=1}^p\sum\limits_{r=0}^{t_j-k}
\alpha_{1,r}^jx_{r}^j, \quad 0\leq k \leq t_p,$$
$$[x_k^1, y_1] = \sum\limits_{j=1}^{p-q}\sum\limits_{r=0}^{t_j-k}
\alpha_{1,r}^jx_{r}^j, \quad 1 \leq q \leq p-1, \ t_{p-q+1}+1 \leq
k \leq t_{p-q}.$$

Consider the equalities $$[[x^1_0, e],y_1] = [x_0^1, [e, y_1]] +
[[x_0^1, y_1], e] = [[x_0^1, y_1], e]=
$$ $$=[\sum\limits_{j=1}^p\sum\limits_{r=0}^{t_j} \alpha_{1,r}^jx_r^j,
e] = \sum\limits_{j=1}^p\sum\limits_{r=1}^{t_j}
\alpha_{1,r}^j(-r(t_j+1-r))x_{r-1}^j.$$

On the other hand $$[[x^1_0, e],y_1] = 0.$$

Comparing the coefficients at the basis elements, we get
$$\alpha_{1,r}^j=0, \quad 1 \leq j \leq p, \ 1\leq r \leq t_j.$$
Thus, we have
$$[x_k^1, y_1] = \sum\limits_{j=1}^p\alpha_{1,j}^1x_{k}^j, \quad 0 \leq k \leq t_p,$$
$$[x_k^1, y_1] = \sum\limits_{j=1}^{p-q}\alpha_{1,j}^1x_{k}^j, \quad 1 \leq q \leq p-1, \ t_{p-q+1}+1 \leq k \leq t_{p-1}.$$

Consider the equalities $$[[x^1_0, h],y_1] = [x^1_0, [h, y_1]] +
[[x^1_0, y_1], h] = [\sum\limits_{j=1}^p\alpha_{1,j}^1x_{0}^j, h]
= \sum\limits_{j=1}^p t_j \alpha_{1,j}^1x_{0}^j.$$

On the other hand $$[[x^1_0, h],y_1] = [t_1x^1_0, y_1] =
t_1\sum\limits_{j=1}^p\alpha_{1,j}^1x_{0}^j.$$

Comparing the coefficient at the basis elements, we have $(t_1 -
t_j)\alpha_{1,j}=0.$ The condition $t_1 \neq t_j$ implies that $\alpha_{1,j}=0$ for all $2 \leq j \leq p.$

Thus, rewriting the index of the coefficients, we obtain
$$[x_k^1, y_1] = \alpha_{1,1}x_{k}^1, \quad 0 \leq k \leq t_1,$$
Similarly we obtain
$$[x_k^i, y_j] = \alpha_{i,j}x_{k}^i, \quad 1 \leq i \leq p, \ 0 \leq k \leq t_i, \ 1 \leq j \leq m.$$

Consider the products
$$[x_k^i,[y_p, y_q]] = [[x_k^i,y_p], y_q] - [[x_k^i,y_q], y_p] =
[\alpha_{i,p}x_k^i,y_q] -
[\alpha_{i,q}x_k^i,y_p]=\alpha_{i,p}\alpha_{i,q}x_k^i -
\alpha_{i,q}\alpha_{i,p}x_k^i =0.$$

From the arbitrariness of elements $\{x_k^i, y_p, y_q\}$ and condition $[S_k,S_k] =S_k$ we have
$[I_1,[S_k,S_k]] = [I_1,S_k] = 0.$

From
$$[I_k, sl_2^1] = [[I_k,S_k], sl_2]\subseteq [I_k, [S_k, sl_2^1]] + [[I_k,sl_2^1], S_k]
= [[I_k,sl_2^1], S_k] \subseteq [I_1, S_k] =0$$ we get
$[I_k, sl_2^1]=0.$

Thus, we obtain
$$[sl_2^{j}+I_{j}, S_k + I_k] = [S_k+I_k, sl_2^{j}+I_{j}] = 0.$$
\end{proof}

Analyzing the proof of the Theorem \ref{t3} we obtain the result, which generalize the Example 3.

\begin{thm} Let $L$ be a semisimple Leibniz algebra such that $L=(sl_2\oplus S)
\dot{+}I$ and $I_1 = [I,sl_2]$ is a reducible over $sl_2.$ Let
$I_1 = I_{1,1}\oplus I_{1,2}\oplus \dots \oplus I_{1,p},$ where
$I_{1,j}$ is an irreducible over $sl_2.$ If $$dim I_{1,j_1} = dim
I_{1,j_2} = \dots dim I_{1,j_s} =t+1,$$ then there exist $(t+1)$
pieces of $s$-dimensional submodules $I_{2,1}, I_{2,2}, \dots
I_{2,t+1}$ of module $I_2=[I, S]$ (i.e. $dim I_{2,i} =s, \ 1 \leq i
\leq t+1$) such that
$$I_{2,1}+ I_{2,2}+ \dots +I_{2,t+1} = I_1 \cap I_2.$$
\end{thm}

\begin{proof} Let $L$ satisfies the condition of theorem.
Without loss of generality we can assume that
$$dim I_{1,1} = dim
I_{1,2} = \dots =dim I_{1,s} =t+1.$$

Analogously as in the proof of
Proposition \ref{prop1} considering the Leibniz identities for the products
$$[[x^1_k, f],y_j] = [x_k^1, [f,
y_j]] + [[x_k^1, y_j], f],$$
$$[[x^1_k, e],y_j] = [x_k^1, [e,
y_j]] + [[x_k^1, y_j], e],$$
$$[[x^1_k, h],y_j] = [x_k^1, [h,
y_j]] + [[x_k^1, y_j], h],$$ we obtain
$$[x_k^i, y_j] = \sum\limits_{r=1}^s \alpha_{j,r}^ix_k^r,$$
where $0\leq k \leq t, \ 1\leq i \leq s, \ 1\leq j \leq m.$

From these products it is not difficult to see that
$$I_{2,j} = \{x_j^1, x_j^2, \dots, x_j^s\}\ (0 \leq j \leq t)$$
are $s$-dimensional submodules of $I_2$ over $S.$
\end{proof}

Finally, we remark that a semisimple Leibniz algebra is decomposed into a direct sum of Lie-simple ones if and only if $I_p\cap I_q =\{0\}$ for any $p \neq q$ (in denotations of Lemma \ref{l1}).

\end{document}